\documentclass[12pt]{amsart}
\usepackage[top=1in, bottom=1in, left=1in, right=1in]{geometry}
\usepackage{amsmath}
\usepackage{amssymb}
\usepackage{graphicx}

\numberwithin{equation}{section}

\newtheorem{theorem}{Theorem}[section]
\newtheorem{lemma}[theorem]{Lemma}
\newtheorem{prop}[theorem]{Proposition}
\newtheorem{cor}[theorem]{Corollary}
\theoremstyle{definition}
\newtheorem{definition}[theorem]{Definition}
\newtheorem{remark}[theorem]{Remark}
\newtheorem{example}[theorem]{Example}

\newcommand{\abs}[1]{\lvert #1 \rvert}

\newcommand{\Aut}{\operatorname{Aut}}

\newcommand{\Aff}{\operatorname{Aff}}

\newcommand{\Z}{\mathbb{Z}}

\newcommand{\zdiv}[1]{{\operatorname{ZDiv}(#1)}}

\newcommand{\im}{\operatorname{im}}
\renewcommand{\ker}{\operatorname{ker}}
\newcommand{\autconj}{\operatorname{conj}}

\renewcommand{\Z}{\mathbb{Z}}

\title{On the Structure of Involutions and Symmetric Spaces of Dihedral Groups}
\author{K. K. A. Cunningham, T. J. Edgar, A. G. Helminck, B. F. Jones, H. Oh, R. Schwell, J.~F.~Vasquez}

\thanks{
The authors thank the American Institute of Mathematics in Palo Alto,
CA for their generous support.
}

\begin{document}
\begin{abstract}
  We initiate the study of analogues of symmetric spaces for the
  family of finite dihedral groups. In particular, we investigate the
  structure of the automorphism group, characterize the involutions of
  the automorphism group, and determine the fixed-group and symmetric
  space of each automorphism.
\end{abstract}

\maketitle

\section*{Introduction}
Let $G$ be a group and $\theta\in\Aut(G)$ such that $\theta^m={\sf id}$. We then define the following two sets:
\begin{align*}
H &= G^\theta=\{g\in G\mid \theta(g)=g\}\\
Q &=\{g\in G\mid g=x\theta(x)^{-1} \mbox{ for some } x\in G\}.
\end{align*}
The set $H$ is the fixed-point subgroup of $\theta$ and $Q$ is known as a {\em Generalized Symmetric Space}. If $\theta$ is an involution and $G$ is a real reductive Lie group, then the set $Q$ is a reductive symmetric space. If $G$ is a reductive algebraic group defined over an algebraically closed field ${\sf k}$, then $Q$ is also known as a symmetric variety and if $G$ is defined over a non-algebraically closed field ${\sf k}$, then the set $Q_{\sf k}:= \{ x\theta(x)^{-1} \mid x\in G_{\sf k} \}$ is called a symmetric ${\sf k}$-variety. Here $G_{\sf k}$ denotes  the set of ${\sf k}$-rational points of $G$. Reductive symmetric spaces and symmetric ${\sf k}$-varieties are well known for their role in many areas of mathematics. They are probably best known for their fundamental role in representation theory  \cite{Helminck94}. The generalized symmetric spaces as defined above are of importance in a number of areas as well, including group theory, number theory, and representation theory. 

One of the first questions that arises in the study of these generalized symmetric spaces is the classification of the automorphisms up to isomorphy, where isomorphy is given by either conjugation by inner automorphisms, outer automorphisms, or both depending on which makes the most sense. 

To analyze the structure of these generalized symmetric spaces one looks at orbits of both the group itself and the fixed point group. 
Both $G$ and $H$ act on $Q$ by $\theta$-twisted conjugation which we denote by $*$. For $g\in G$ and $q\in Q$ we have
$$g*q=gq\theta(g)^{-1}.$$
Since $H$ is the set of fixed points, this action is simply conjugation when restricted to $H$. Also, if $\theta$ is an involution, then $Q\cong G/H$ under the map $\tau : G\to Q$ given by $\tau(g)=g\theta(g)^{-1}$. 

We are interested in classifying $G$ and $H$ orbits in $Q$ and describing how the $G$-orbits decompose into $H$-orbits. In particular, if $\theta$ is an involution it is known that $H\backslash Q\cong H\backslash G/H$. These orbits and double cosets  play an important role in representation theory.

%When $G$ is an algebraic group defined over a field $\F$ and $\theta\in\Aut(G)$
%such that $\theta^m={\sf id}$, there is interest in classifying such
%$\theta$ up to equivalence, where equivalence is given by conjugation
%by inner automorphisms, outer automorphisms, or both depending on
%which makes the most sense. Once this is done, it is also interesting
%to study the following two sets:
%\begin{align*}
%H &= G^\theta=\{g\in G\mid \theta(g)=g\}\\
%Q &=\{g\in G\mid g=x\theta(x)^{-1} \mbox{ for some } x\in G\}.
%\end{align*}
%The set $H$ is the fixed-point subgroup of $\theta$ and $Q$ is known
%as a symmetric $\F$-variety \cite{Helminck94}. Symmetric $\F$-varieties play an important
%role in representation theory and harmonic analysis \cite{}.
%
%Both $G$ and $H$ act on $Q$ by $\theta$-twisted conjugation which we
%denote by $*$. This action is defined as follows: for $g\in G$ and
%$q\in Q$, let:
%$$g*q=gq\theta(g)^{-1}.$$
%Since $H$ is the set of fixed points of $\theta$, this action is simply
%conjugation when restricted to $H$. Also, if $\theta$ is an
%involution, then $G/H \cong Q$ under the map $\tau: G \to Q$ given by
%$\tau(g)=g\theta(g)^{-1}$ since $H = \ker \tau$.
%
%A further interesting problem is to classify the $G$ and $H$ orbits in $Q$
%and determine how the $G$-orbits decompose into $H$-orbits. In particular, if
%$\theta$ is an involution it follows from the previous paragraph that
%$H\backslash Q \cong H\backslash G/H$. These orbits (double cosets) also play an
%important role in representation theory.

This paper focuses on understanding the structures described above when $G$ is the dihedral group of order $2n$, that is $G=D_n$. 
In particular, we determine the sets $H$ and $Q$ explicitly, provide a procedure for enumerating the involutions, and give a closed formula for counting the equivalence classes of involutions of $D_n$.

The paper is organized as follows.
In Section 1, we introduce the necessary preliminaries including our choices for notation and the
well known description of the automorphism group of $D_n$ as well as
some relevant examples. In Section 2, we investigate
and describe all the automorphisms of $D_n$ of a fixed order. We recall the notion
of equivalence of automorphisms and give simple
conditions for two automorphisms of $D_n$ to be equivalent. Furthermore, we provide a formula for computing the total number of equivalence classes of automorphisms of a fixed order. In Section
3, we give full descriptions of the sets $H$ and $Q$ as well as the orbits of $Q$ under
the action by $H$ and by $G$. In Section 4, we find stronger results for the involutions in the automorphism group. Our main technical result is to partition the set of involutions in such a way that there are at most two distinct equivalence classes of involutions in each piece of the partition (cf. Theorem \ref{thm:equivinvols}). We also show that in the involution case, $Q$ is always a subgroup of $G$, and we describe the subgroup structure of $Q$. Furthermore, we introduce the set of twisted involutions, $R$, and give a characterization of this set and its relation to the generalized symmetric space. Using our results on involutions of $D_n$, we provide a counterexample to
a standard theorem on equivalence of involutions which holds for
algebraic groups. Finally, in Section 5, we complete the discussion by applying our methods to the infinite dihedral group in order to understand the automorphisms of finite order, as well as $H$ and $Q$ in that context.

\section{Preliminaries}

\subsection{The Dihedral Group and its Automorphism Group}
Throughout the paper we denote by $D_n$ the group of symmetries of the
regular $n$-gon. More specifically, $D_n$ is a finitely generated
group given by the presentation
$$D_n=\langle r,s\mid r^n=s^2=1,\ sr=r^{-1}s\rangle.$$
We use this presentation, instead of the presentation as a Coxeter
group (see \cite[Example 1.2.7]{Bjorner},
\cite[\S1.1, \S4.2]{Humphreys}, \cite[\S1.2]{Bourbaki}), because it is convenient for describing
the automorphism group of $D_n$. From the presentation it is clear
that
\[ D_n = \{ 1, r, r^2, \ldots, r^{n-1}, s, rs, r^2s, \ldots, r^{n-1}s
\} \]
and we say that an element of $D_n$ is presented in \emph{normal form}
if it is written as $r^ks^m$ for some integers $0 \le k < n$ and $s
\in \{0, 1\}$.

Throughout the paper, $\Z_n$ denotes the additive group of integers
modulo $n$ and $U_n$ denotes the multiplicative group of units of
$\Z_n$.  Recall that $a \in U_n$ if and only if $(\dot{a},n) = 1$ for
any representative integer $\dot{a}$ of $a \in \Z_n$.

We also find it useful to have terminology for the $k$-th roots of unity in $\Z_n$ (equivalently in $U_n$), which can be described by
$$\mathcal{R}_n^k:=\{a\in U_n\mid a^k=1\}.$$

The automorphism group of $D_n$ is well known (see \cite[Theorem A]{Walls}) 
and is described in the following lemma.
\begin{lemma}
\label{lem:aut} The automorphism group of $D_n$ is isomorphic to the 
group of affine linear transformations of $\Z_n$:
$$\Aut(D_n)\cong \Aff(\Z_n)=\left\{ ax+b: \Z_n\to\Z_n \mid a \in U_n, b \in
\Z_n \right\}$$
and the action of $ax+b$ on elements of $D_n$ in normal form is given by:
\[ (ax+b).(r^k s^m) = r^{ak+bm}s^m.\]
\end{lemma}
\begin{proof}
Observe that $\theta(r)$ must have order $n$ since $r$ does. The
elements of $D_n$ of order $n$ are precisely the $\{ r^a \mid a \in
U_n \}$.
So let $\theta(r) = r^a$ for some $a$ relatively prime to
$n$. Similarly, $\theta(s)=r^bs$ for some $0 \le b < n$ since $s$ has
order 2 and the only elements of $D_n$ of order 2 are of the form $r^bs$.
Putting these together, we see that if
$\theta\in\Aut(D_n)$, then for any $r^ks^m\in D_n$ we must have
$\theta(r^ks^m)=r^{ak+bm}s^m$. There is now a clear map from
$\Aut(D_n)\to \Aff(\Z_n)$ and it is easily checked to be an isomorphism.
\end{proof}

Throughout the paper we abuse notation and write $\theta = ax+b$ for
elements $\theta \in \Aut(D_n)$ according to Lemma \ref{lem:aut}.

%\begin{remark}
%\label{rmk:coxeter-automorphisms}
%Note that Lemma \ref{lem:aut} describes the group automorphisms of
%$D_n$ and not the more restricted set of \emph{Coxeter group
%  automorphisms}, i.e. group automorphisms that preserve the Coxeter
%system. The structure of the Coxeter group automorphisms is known in
%general, see \cite{Tits1988}.
%\end{remark}

\begin{example}
  Consider $\theta = \autconj(g)$,
  conjugation by an element $g\in D_n$. 
  If $\theta = ax + b$ from Lemma \ref{lem:aut}, and
  if $g=r^k$, then it is easy to see that $a=1$ and
  $b=2k$ (in $U_n$ and $\Z_n$ respectively). Also, if $g=r^ks$, then we have $a=n-1$ and
  $b=2k$. In particular, if $\theta=ax+b$ is an \emph{inner automorphism},
  then (abusing notation) we have $a\equiv\pm 1 \pmod{n}$. Also, we see that when $n$ is even, there are $n$
  distinct inner automorphisms (corresponding to $a\equiv\pm 1\pmod{n}$ and $b\in\langle2\rangle\leq
  \Z_n$, and when $n$ is odd, there are $2n$ distinct inner automorphisms
  corresponding to $a\equiv\pm 1\pmod{n}$ and $b\in\langle2\rangle=\Z_n$.  Finally, note
  that the identity automorphism of $D_n$
  corresponds to $x \in \Aff(\Z_n)$.
\end{example}

\begin{example}
\label{coxeter-example}
  It is well known from the theory of Coxeter groups that when $n$ is
  even, there is an outer automorphism of $D_n$ corresponding to the
  interchanging of the two conjugacy classes of reflections, also
  known as the diagram automorphism (cf. \cite[\S4.2]{Bourbaki}). This
  automorphism is described by $\theta=ax+b$ where $a=n-1$ and
  $b=n-1$. The previous example shows that this automorphism
  is inner when $n$ is odd, and outer when $n$ is even.
\end{example}

\section{Automorphisms of $D_n$}
\label{sec:aut}

In this section we describe the action of the automorphism group on
$D_n$, we characterize the automorphisms of fixed order, and we discuss the notion of equivalent automorphisms. 

For any $c\in\Z_n$, we define $\zdiv{c}=\{y\in\Z_n\mid cy\equiv 0
\mbox{ mod } n\}$. It is trivial to check that $\zdiv{c}$ is a
subgroup of $\Z_n$. In particular, $\zdiv{c}=\ker\pi_{c}$ where
$\pi_{c}:\Z_n\to\Z_n$ given by $\pi_{c}(x)=cx$. Note that
$\im(\pi_{c})=\langle c\rangle$ and so
$|\im(\pi_{c})|=\frac{n}{\gcd(c,n)}$. Since $\im(\pi_{c})\cong
\Z_n/\ker(\pi_{c})\cong \Z_n/\zdiv{c}$, it follows that
$|\zdiv{c}|=\gcd(c,n)$. We use this notation to describe the automorphisms of finite order dividing $k$. 
\begin{prop}
\label{prop:orderk}
Let $k\geq 1$ be an integer and $\theta\in\Aut(D_n)$ with
$\theta=ax+b$. Then $\theta^k={\sf id}$ if and only if $a\in\mathcal{R}_n^k$ and $b\in \zdiv{a^{k-1}+a^{k-2}+\cdots +a+1}$.
\end{prop}
\begin{proof}
Straightforward computation in $\Aut(D_n)$ gives us that 
$$\theta^k=a^kx+(a^{k-1}+a^{k-2}+\cdots +a+1)b$$
Since we have identified the identity automorphism with $x \in \Aff(\Z_n)$
the following two equations hold:
\begin{equation}
\label{congruences}
\begin{aligned}
a^k & \equiv 1 \pmod{n} \\
(a^{k-1} + a^{k-2} + \cdots + 1) b & \equiv 0 \pmod{n}.
\end{aligned}
\end{equation}
In the notation described above, these are evidently equivalent to $a\in\mathcal{R}_n^k$ and\\ 
$b\in~\zdiv{a^{k-1}+a^{k-2}+\cdots +a+1}$.
\end{proof}

In what follows, we let $\Aut_k(D_n)=\{\theta\in\Aut(D_n) \mid \theta^k={\sf id}\}\subseteq \Aut(D_n)$.
\begin{prop}
\label{prop:numauts}
For any $n$ and $k\geq 1$, we have
$$|\Aut_k(D_n)|=\sum_{a\in \mathcal{R}_n^k}\gcd(a^{k-1}+a^{k-2}+\cdots+a+1,n).$$
\end{prop}

\begin{proof}
This follows from Proposition \ref{prop:orderk}: for any $a\in\mathcal{R}_n^k$ there are 
$$\abs{\zdiv{a^{k-1}+a^{k-2}+\cdots+1}}=\gcd(a^{k-1}+a^{k-2}+\cdots+1,n)$$ elements $b$ such that
$(a^{k-1}+a^{k-2}+\cdots+1)b\equiv 0$ (mod $n$), and every automorphism $\theta\in\Aut_k(D_n)$ must be of this form.
\end{proof}

\begin{definition} Let $\theta_1, \theta_2\in\Aut(D_n)$. 
We say that $\theta_1$ is
\textit{equivalent} to $\theta_2$ and write $\theta_1\sim\theta_2$ if
and only if they are conjugate to each other, i.e. if there is
$\sigma\in\Aut(D_n)$ with $\sigma\theta_1\sigma^{-1}=\theta_2$. For any $\theta\in\Aut(D_n)$ we let $\overline{\theta}=\{\sigma\in\Aut(D_n)\mid \theta\sim \sigma\}$ be the equivalence class of $\theta$. 
\end{definition}

\begin{remark}
This definition of equivalence is broad (since we allow conjugation by
any automorphism, not just the inner ones), but still useful. In
particular, it simplifies the statement of the following proposition
and it allows us not to worry about the parity of $n$ in several places.
\end{remark}

\begin{prop}
\label{prop:equivalence}
Let $\theta_1=ax+b\in\Aut(D_n)$ and $\theta_2=cx+d\in\Aut(D_n)$. Then
$\theta_1\sim\theta_2$ if and only if $a=c$ and $fb-d\in\langle
a-1\rangle\leq \Z_n$ for some $f\in U_n$.
\end{prop}
\begin{proof}
  Suppose that $\sigma=fx+g\in\Aut(D_n)$ (so that $f\in U_n$ and $g\in\Z_n$). It is easily checked that $\sigma^{-1}=f^{-1}x-f^{-1}g$
  where $f^{-1}\in\Z_n$ since $f\in U_n$. Then
  $\sigma\theta_1\sigma^{-1}=\tau$, where
  $\tau=f(a(f^{-1}x-f^{-1}g)+b)+g=ax+fb-g(a-1)$. Thus, we can conclude
  that $\sigma\theta_1\sigma^{-1}=\theta_2$ if and only if $a=c$ and
  $fb-g(a-1)= d$. The second equation can be written as $fb-d=g(a-1)$, and thus $fb-d\in\langle
  a-1\rangle$.
\end{proof}

The previous proposition allows us to describe the conjugacy classes of automorphisms with order dividing $k$. 

\begin{prop}
\label{lem:equivalence}
Suppose $n$ is fixed and let $k\geq 1$.
\begin{enumerate}
\item
For any $a\in\mathcal{R}_n^k$, $\langle a-1\rangle\leq \zdiv{a^{k-1}+a^{k-2}+\cdots+a+1}$.
\item
For any $a\in\mathcal{R}_n^k$, $U_n$ acts on the cosets $\zdiv{a^{k-1}+a^{k-2}+\cdots+a+1}/\langle a-1\rangle$.
\item
The set $\Aut_k(D_n)$ is partitioned into equivalence classes indexed by pairs $(a,B)$ where $a\in\mathcal{R}_n^k$ and $B$ is an orbit of $U_n$ on \hbox{$\zdiv{a^{k-1}+a^{k-2}+\cdots +a+1}/\langle a-1\rangle$}.
\end{enumerate}
\end{prop}

\begin{proof}
Part (1) is trivial since $(a-1)(a^{k-1}+a^{k-2}+\cdots +a+1)=a^k-1\equiv 0$. For part (2), we simply recognize that $U_n=\Aut(\Z_n)$ acts on $\Z_n$ by multiplication and since every subgroup is cyclic, $U_n$ must stabilize the subgroups of $\Z_n$. Finally, part (3) is simply Proposition \ref{prop:equivalence} in terms of the $U_n$-action on $\zdiv{a^{k-1}+a^{k-2}+\cdots +a+1}/\langle a-1\rangle$.\end{proof}

Proposition \ref{prop:equivalence} and Proposition \ref{lem:equivalence} help us determine the total number of equivalence classes of finite order automorphisms. For fixed $a\in\mathcal{R}_n^k$, we want to compute 
\begin{equation}
\label{eq:numequiv}
N_a:=|\{\overline{\theta}\mid \theta=ax+b\in\Aut_k(D_n)\}|.
\end{equation}

According to Proposition \ref{lem:equivalence}, given $a\in\mathcal{R}_n^k$, computing $N_a$ amounts to counting the number of orbits of the $U_n$-action on $\zdiv{a^{k-1}+a^{k-2}+\cdots +a+1}/\langle a-1\rangle$. 

\begin{theorem}
\label{unorbits}
Let $n$ be fixed and let $a\in\mathcal{R}_n^k$. Then, the number of orbits of $U_n$ on $\zdiv{a^{k-1}+a^{k-2}+\cdots +a+1}/\langle a-1\rangle$ is equal to the number of divisors of 
$$\frac{\gcd(a-1,n)\gcd(a^{k-1}+a^{k-2}+\cdots+a+1,n)}{n}.$$
\end{theorem}

\begin{proof}
The $U_n$-orbits on $\Z_n$ are indexed by the subgroups of $\Z_n$. Thus, the $U_n$-orbits on $\zdiv{a^{k-1}+a^{k-2}+\cdots +a+1}/\langle a-1\rangle$ are indexed by subgroups $L\leq \Z_n$ such that $\langle a-1\rangle\leq L\leq \zdiv{a^{k-1}+a^{k-2}+\cdots +a+1}$. It is well known that the subgroup lattice of $\Z_n$ is isomorphic to the divisor lattice of $n$. Under the previous lattice isomorphism, the subgroup $\langle a-1\rangle$ corresponds to the divisor $\gcd(a-1,n)$ and the subgroup $\zdiv{a^{k-1}+a^{k-2}+\cdots +a+1}$ corresponds to the divisor $$\frac{n}{\gcd(a^{k-1}+\cdots+a+1,n)},$$ and the subgroups between these groups correspond to the divisors of $n$ between $\gcd(a-1,n)$ and  $\frac{n}{\gcd(a^{k-1}+\cdots+a+a,n)}$ in the divisor lattice. Finally, it is known that this sub-lattice of the divisors of $n$ is isomorphic to the divisor lattice of $$\frac{\gcd(a-1,n)}{\frac{n}{\gcd(a^{k-1}+\cdots+a+1,n)}}=\frac{\gcd(a-1,n)\gcd(a^{k-1}+a^{k-2}+\cdots+a+1,n)}{n}.$$ 
\end{proof}

Given a specific $n$, Theorem \ref{unorbits} and Proposition \ref{lem:equivalence} allow us to compute all the equivalence classes of automorphisms of order $k$ using any standard computer algebra package. In section 4, we investigate $N_a$ when $a\in\mathcal{R}_n^2$ (i.e. $\theta$ is an involution).

%%%%%%%%%%%%%%%%%%%%%%%%%%%%%%%%%%%%%%%%%%%%%%%%%%%%%%%%%%%%%%%

\section{Fixed Groups and Symmetric Spaces of Automorphisms}
\label{sec:results-aut}

Recall from the introduction that we are
interested in two different subsets of a group $G$. Namely, given an
automorphism, $\theta$, we want to compute

$$
\begin{aligned}
H_\theta&=G^\theta=\{x\in G\mid \theta(x)=x\}\ \mbox{and}\\
Q_\theta&=\{g\in G\mid g=x\theta(x)^{-1} \mbox{ for some } x\in G\}.
\end{aligned}$$

When $\theta$ is understood to be fixed, we drop the subscript from
our notation. The following theorem characterizes these spaces in the
case of dihedral groups.

\begin{theorem}
\label{thm:hqr}
Let $G=D_n$ and $\theta=ax+b\in\Aut(D_n)$ be of finite order, and let
$H$ and $Q$ be as defined above. Then
$$
\begin{aligned}
H&=\{r^k\mid k(a-1)\equiv 0 \mbox{ (mod $n$)}\}\cup \{r^ks\mid k(a-1)\equiv -b \mbox{ (mod $n$)}\}\ \mbox{and}\\
Q&=\{r^k\mid k\in\langle a-1\rangle\cup(-b+\langle a-1\rangle)\}.
\end{aligned}
$$
\end{theorem}

\begin{proof}
  Recall that if $\theta=ax+b\in\Aut(D_n)$, then the formula
  $\theta(r^ks^m)=r^{ak+bm}s^m$ provides the full description of the
  action of $\theta$. All of the results of the theorem arise from the
  definitions of $H$ and $Q$ using the action described above by
  $\theta$. We demonstrate the computation for $H$ in what
  follows, and $Q$ is similar. Suppose $\theta(r^k)=r^k$, then $r^{ak}=r^k$ and so
  $ak-k\equiv 0$ (mod $n$), i.e. $k(a-1)\equiv 0$ (mod $n$). On the
  other hand, if $\theta(r^ks)=r^ks$, then $r^{ak+b}s=r^ks$ and so
  $ak-k+b\equiv 0$ (mod $n$), i.e. $k(a-1)\equiv -b$ (mod $n$).
\end{proof}

Using the descriptions of $H$ and $Q$, we obtain further results as well.

\begin{cor}
\label{h-structure}
Let $\theta=ax+b\in\Aut(D_n)$ be of finite order. Then
\begin{enumerate}
\item
If $b\not\in\langle a-1\rangle$, then $H\cong\zdiv{a-1}$ is cyclic.
\item
If $b\in\langle a-1\rangle$, then $H\cong\zdiv{a-1} \rtimes \Z_2$ is dihedral.
\end{enumerate}
\end{cor}

\begin{cor}
\label{cor:qgroup}
Let $n$, $k$, and $\theta=ax+b\in\Aut_k(D_n)$ be fixed. If $b\in\langle a-1\rangle$ then $Q$ is a subgroup of $D_n$ and $Q\cong\langle a-1\rangle$.\end{cor}

\begin{proof}
Since $b\in\langle a-1\rangle$, then it follows from Theorem \ref{thm:hqr} that $Q=\{r^k\mid k\in\langle 
a-1\rangle\}$.
\end{proof}

We will show in the next section that when $\theta$ is an involution, $Q$ is always a subgroup of $D_n$ 
(see Corollary \ref{cor:Q-is-ZDiv}).

\begin{cor}
\label{cor:HQ-is-G}
Let $\theta=ax+b\in\Aut(D_n)$ be a fixed automorphism. If $b\not\in\langle a-1\rangle$, then $HQ\ne D_n$. If $b\in\langle a-1\rangle$, then the following are equivalent
\begin{enumerate}
\item
$HQ=D_n$
\item
$H\cap Q=\{1\}$ in $D_n$, and
\item
$\gcd(a-1,n)$ is relatively prime to $\frac{n}{\gcd(a-1,n)}$.
\end{enumerate}
\end{cor}

\begin{proof}
From Corollary \ref{h-structure} and the fact that $Q\subseteq\langle r\rangle$, we see that if $b\not\in\langle a-1\rangle$ then $HQ
\subseteq\langle r\rangle\ne D_n$. If $b\in\langle a-1\rangle$ then $H$ contains a reflection and so $HQ = D_n$ if and only if 
$r\in HQ$. From Corollaries \ref{h-structure} and \ref{cor:qgroup}, we see that $r\in HQ$ if and only if the subgroups 
$\zdiv{a-1}$ and $\langle a-1\rangle$ in $\Z_n$ have relatively prime generators. 
This last condition is clearly equivalent to both (2) and (3).
\end{proof}

\begin{example} \label{autexamples} Let $G=D_{36}$. We illustrate Theorem \ref{thm:hqr} and its corollaries for three different 
automorphisms of $D_{36}$. First, consider the automorphism $\theta_1=19x+18$.  Note that $\theta_1$ is an involution. Applying 
Theorem \ref{thm:hqr} yields that for $\theta_1$, $H=H_1=\{1,r^2,r^4,\ldots,r^{34}, r^3s,r^5s,\ldots, r^{35}s\}$ and $Q=Q_1=\{1,r^
{18}\}$, which gives that $H_1\cap Q_1=\{1,r^{18}\}=Q_1$. Observe that the statement of Corollary \ref{cor:qgroup} holds in this 
case: $H_1 Q_1\neq D_{36}$. This follows from the fact that $18\in\langle 18\rangle$, but $\gcd(18,36)=18$ is not relatively prime to 
$36/\gcd(18,36)=2$.  

Now consider $\theta_2=5x+2$ which is an automorphism of order 6. In this case, $H_2=\{1,r^9,r^{18},r^{27}\}$ and $Q_2=
\{1,r^2,r^4,\ldots,r^{34}\}$. Thus, $H_2\cap Q_2=\{1, r^{18}\}$ and $H_2 Q_2=D_{36}$.  However, in this case it is because $2\not\in
\langle 4\rangle$.  

Finally, consider the automorphism $\theta_3 = 5x + 4$ of $D_{36}$. This is also an automorphism of order 6. In this case, $H_3 = 
\{1, r^9, r^{18}, r^{27},  r^8s, r^{17}s, r^{26}s, r^{35}s\}$ and $Q_3 = \{1, r^4, r^8, r^{12}, \ldots , r^{32}\}$. Thus, $H_3 \cap 
Q_3 = \{1\}$ and this agrees with Corollary \ref{cor:qgroup} since $r^9r^{28} = r$ and so $H_3 Q_3 = D_{36}$.  
%Consider the automorphism $\theta_1 = 15x+7$ of $D_{28}$. Note that this is in fact an
%involution, as it satisfies the congruence requirements in Proposition \ref{prop:orderk} for $k=2$ (and
%is not equal to the identity map). Applying Theorem \ref{thm:hqr} yields that $H = \{1, r^2, r^4, \dots, r^{26}\}$ and $Q = \{1, 
%r^7, r^{14}, r^{21}\}$, which gives that $H \cap Q = \{1, r^{14}\}$, so by Corollary \ref{qgroup}, $HQ\neq D_{28}$.  
\end{example}

From Theorem \ref{thm:hqr}, we see that $H$ is the disjoint union of $\{r^k\mid k(a-1)\equiv 0 \pmod{n}\}$ and $\{r^ks\mid k
(a-1)\equiv -b\pmod{n}\}$. Notice that the second set may be empty if there is no solution, $i$, to the equation $i(a-1)\equiv -b
\pmod{n}$ for fixed $a$ and $b$ (i.e. if $b\not\in\langle a-1\rangle$). In fact, the two possibilities for the $H$-orbits on $Q$ are 
determined by the existence of such a solution.

\begin{prop} \label{orbits}
Let $G=D_n$ and $\theta=ax+b\in\Aut_k(D_n)$ for some $k$. 
\begin{enumerate}
\item
If $b\not\in\langle a-1\rangle$, then the $H$-orbits on $Q$ are:
$$H\backslash Q=\left\{\{r^j\}\mid j\in \langle a-1\rangle\cup(-b+\langle a-1\rangle)\right\}.$$
\item
If $b\in\langle a-1\rangle$, then the $H$-orbits on $Q$ are:
$$H\backslash Q=\left\{\{r^j, r^{-j}\}\mid j\in \langle a-1\rangle\right\}.$$
\end{enumerate}
In either situation $G\backslash Q=\{Q\}$, i.e. there is a single $G$-orbit on $Q$. 
\end{prop}

\begin{proof}
Since $H$ is fixed by $\theta$, the action of $H$ on $Q$ is simply by conjugation. We note that $Q\subseteq \langle r\rangle\leq D_n$ 
and so we only need to describe the action of $H$ on a general element, $r^j$. Then, let $r^i\in\{r^k\mid k(a-1)\equiv 0 \pmod{n}\}
\subseteq H$. We see $r^ir^jr^{-i}=r^j$ so $\langle r\rangle$ fixes $Q$ pointwise. Now suppose $r^is\in\{r^ks\mid k(a-1)\equiv -b
\pmod{n}\}\subseteq H$. We have $r^isr^j(r^is)^{-1}=r^{-j}$ and so $r^is$ takes $r^j$ to $r^{-j}$. The result now follows since we 
will have $\{r^j,r^{-j}\}$ as an orbit if and only if there is $r^is\in H$, which is true if and only if $b\in\langle a-1\rangle$.

Finally, we demonstrate that every element of $Q$ is in the $G$-orbit of $1\in Q$. Notice that every element of $Q$ is either of the 
form $q_1=r^{p(a-1)}\in Q$ or $q_2=r^{-b+p(a-1)}\in Q$. We have $r^{-p}\in G$ and $r^{-p}1\theta(r^{-p})^{-1}=r^{p(a-1)}=q_1$. 
Additionally, we have $r^{-p}s\in G$ and $r^{-p}s1\theta(r^{-p}s)^{-1}=r^{-p}r^{ap-b}=r^{-b+p(a-1)}=q_2$. Therefore, for any $q\in Q
$, we can find $g\in G$ such that $g1\theta(g)^{-1}=q$ and so $G*1=Q$. Hence $G\backslash Q=\{Q\}$.
\end{proof}

\begin{example}
Revisiting $\theta_1$, $\theta_2$, and $\theta_3$ of Example \ref{autexamples}, we apply Proposition \ref{orbits} to obtain that for 
$\theta_1$, since $18\in \langle 18\rangle$ and $r^{18} = r^{-18} \in D_{36}$, $$H_1\backslash Q_1=\{\{1\},\{r^{18}\}\};$$ for $
\theta_2$, since $2\not\in\langle 4\rangle$, $$H_2\backslash Q_2=\{\{1\},\{r^2\},\{r^4\},\ldots,\{r^{34}\}\};$$ lastly, for $
\theta_3$, since $4\in\langle 4\rangle$, $$H_3\backslash Q_3 = \{\{1\}, \{r^4,r^{-4}\}, \{r^8,r^{-8}\}, \ldots, \{r^{32},r^{-32}\}\}.
$$
\end{example}

%$7\not\in\langle 14\rangle$, $H\backslash Q=\{\{1\},\{r^7\},\{r^{14}\},\{r^{21}\}\}$ and for $\theta_2$, since $4\in \langle 
%4\rangle$, $H\backslash Q=\{\{1\},\{r^4,r^{-4}\},\{r^8,r^{-8}\},\dots,\{r^{32},r^{-32}\}\}$.\\ 

All of the results in this section hold for automorphism of arbitrary finite order $k$, but in the next section we demonstrate that 
we can say more when we restrict our attention to $k=2$, the involutions. In Section 5, we describe which of these results hold in 
the infinite dihedral group.

\section{Involutions in $\Aut(D_n)$}
\label{sec:results-invol}

In this section, we utilize the results from sections 2 and 3 and expand
on them in the situation when $\theta$ is an involution in $\Aut(D_n)$.
For this entire section, we assume that $\theta=ax+b\in\Aut(D_n)$ and
$\theta^2={\sf id}$. In particular, Proposition \ref{prop:orderk} forces
$a\in\mathcal{R}_n^2$ and $b\in\zdiv{a+1},
\mbox{ i.e.}\ (a+1)b\equiv 0\pmod{n}$. The discussion prior to Proposition
\ref{prop:orderk} describes how to determine the total number of involutions
in $\Aut(D_n)$. Recall that $\Aut_2(D_n):=\{\theta\in\Aut(D_n)\mid
\theta^2={\sf id}\}$ is the collection of involutions along with the identity automorphism.

For our discussion of the involutions in $\Aut(D_n)$, it
is necessary to understand the structure of $\mathcal{R}_n^2$, the square
roots of unity in $\Z_n$. The following results can be found in an
elementary number theory text (e.g.  \cite[Example 3.18]{Jones-Jones})
but we include them here for completeness and for the usefulness of the
proof.

\begin{theorem}
\label{thm:squareroots1}
Suppose that $n\geq 1$ and $n=2^mp_1^{r_1}\cdots p_k^{r_k}$ where the $p_i$ are distinct odd primes. Then,
$$|\mathcal{R}_n^2|=\begin{cases}
2^k & \mbox{ if $m=0,1$}\\
2^{k+1} & \mbox{ if $m=2$}\\
2^{k+2} & \mbox{ if $m\geq 3$}.
\end{cases}$$
\end{theorem}

\begin{remark}
\label{rmk:computesqr}
We include the proof of Theorem \ref{thm:squareroots1} in Appendix \ref{sec:appendix} because it demonstrates
exactly how to construct the square roots of unity in $\Z_n$ provided we have
the prime factorization of $n$. The construction only involves the 
Euclidean Algorithm and a map provided by the Chinese Remainder
Theorem. One can use the procedure described in the proof to
effectively compute the square roots of unity in $\Z_n$ for any $n$. The construction of the square roots of unity provided in the appendix also provides the understanding necessary for Corollary \ref{cor:equivclassformula} below.
\end{remark}

The following corollary is the statement of Proposition \ref{prop:numauts} for involutions.
\begin{cor}
\label{cor:numinvols}
For any $n \ge 1$, 
$$|\Aut_2(D_n)|=\sum_{a\in \mathcal{R}_n^2}\gcd(a+1,n).$$
\end{cor}

\begin{remark}
 Corollary \ref{cor:numinvols} together with the proof of Theorem
  \ref{thm:squareroots1} (see Remark \ref{rmk:computesqr}) gives an easy way
  to compute the total number of involutions in $\Aut(D_n)$. Figure
  \ref{fig:numinvolutions} provides a graph of the number of involutions in
  $D_n$ for $n\leq 200$. This plot was generated using Sage \cite{sage}.
\end{remark}

\begin{figure}
\centering
\includegraphics[width=5.5in]{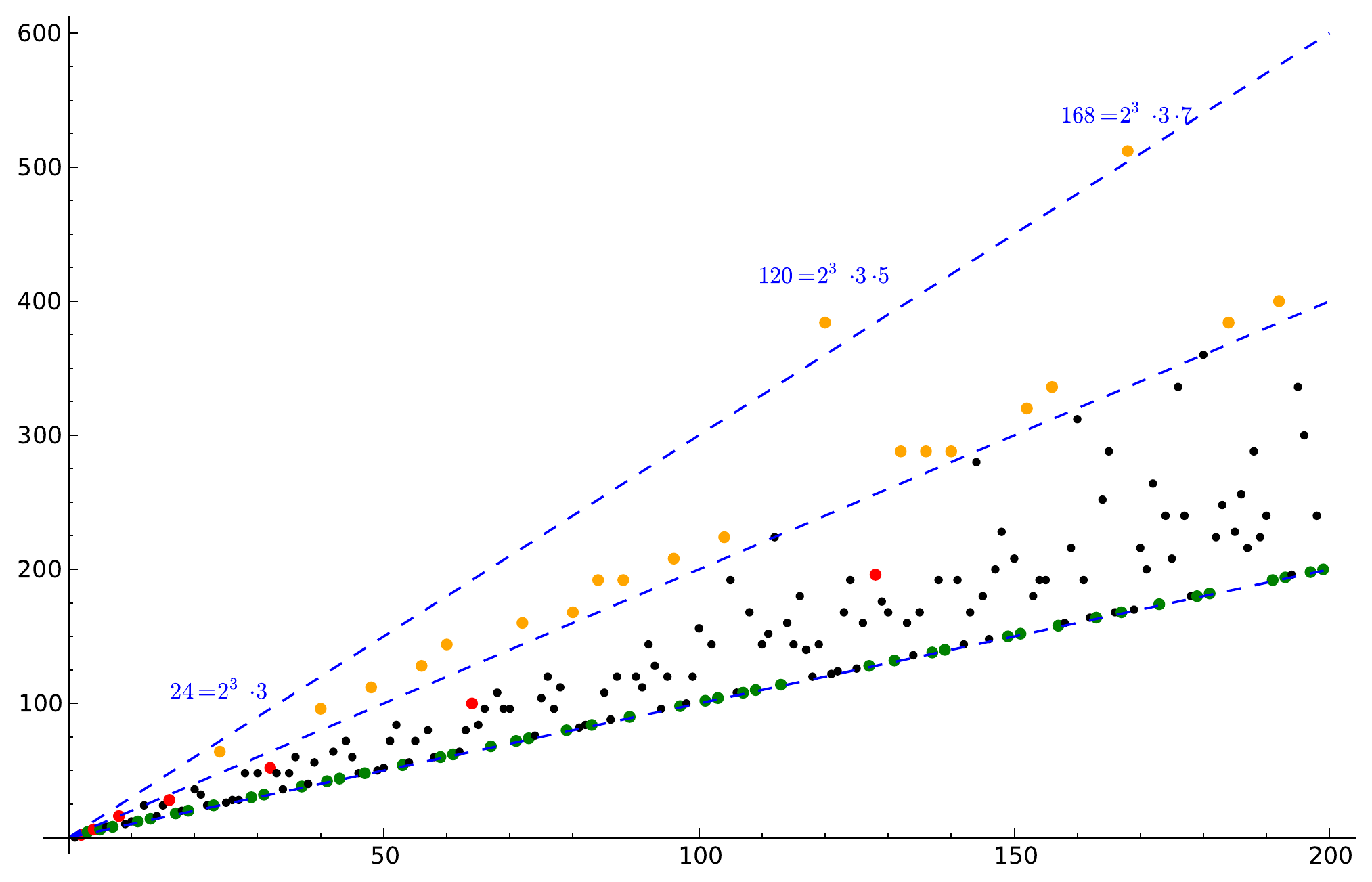}
\caption{The number of involutions in $\Aut(D_n)$ for $n\leq 200$. The
  dashed lines represent the lines with slope 1, 2, and 3.}
\label{fig:numinvolutions}
\end{figure}

Now we proceed to determine the equivalence classes of involutions.
Let $\theta_1 = ax+b$ and $\theta_2=cx+d$. We
know from Proposition \ref{prop:equivalence} that
if $\theta_1 \sim \theta_2$, then
$a=c$. Thus, to determine the equivalence classes of involutions it 
suffices to fix an $a \in \mathcal{R}_n^2$ and determine when $ax+b \sim ax+d$. Furthermore, according to Proposition \ref{lem:equivalence}, with $a$ fixed, we simply need to describe the action of $U_n$ on $\zdiv{a+1}/\langle a-1\rangle$ in order to describe all the equivalence classes.

Recall that we denote the equivalence class of an
involution $\theta$ by $\overline{\theta}$ and the number of 
equivalence classes with fixed leading coefficient $a$ by:
\[ N_a:=\left|\left\{ 
\overline{\theta} \mid 
\theta = ax + b \in \Aut_2(D_n)
\right\}\right|.\]

\begin{theorem}
\label{thm:equivinvols}
Let $a\in\mathcal{R}_n^2$. The following hold.
\begin{enumerate}
\item 
$\langle a-1\rangle\leq \zdiv{a+1}$ and 
$|\zdiv{a+1}/\langle a-1\rangle|\leq 2$.
\item $N_a = \left| \zdiv{a+1}/\langle a-1 \rangle \right| \leq 2 $. 
\item $N_a=2$ if and only if $n=2^m\cdot k$ where $k\geq1$ is odd, $m
  > 0$, and 
\hbox{$a \equiv \pm 1 \pmod{2^m}$.}
\end{enumerate}
\end{theorem}

\begin{proof}
The first part of (1) is simply a restatement of Proposition \ref{lem:equivalence} part (1). Now, suppose that $|\zdiv{a+1}/\langle a-1\rangle|=j$. It is well known that $|\langle
  a-1\rangle|=\frac{n}{\gcd(a-1,n)}$ and we have already observed that
  $|\zdiv{a+1}|=\gcd(a+1,n)$. Thus we have that
  $\frac{\gcd(a+1,n)}{\frac{n}{\gcd(a-1,n)}}=j$, or
  $\gcd(a+1,n)\gcd(a-1,n)=jn$. Now, suppose that
$$n=2^mp_1^{r_1}\cdots p_k^{r_k}$$ 
with $m\geq 0$, the $p_i$ distinct odd primes, and $r_i\geq 1$. Then, we must have that 
\begin{itemize}
\item[]
$\gcd(a+1,n)=2^{m'}p_1^{s_1}\cdots p_k^{s_k}$ with $m'\leq m$ and $0\leq s_i\leq r_i$ for all $i$, and
\item[]
$\gcd(a-1,n)=2^{m''}p_1^{t_1}\cdots p_k^{t_k}$ with $m''\leq m$ and $0\leq t_i\leq r_i$ for all $i$.
\end{itemize}
Now, we note that for all $i\in\{1,...,k\}$ either $s_i=0$ or
$t_i=0$. Indeed, if $s_i>0$ and $t_i>0$ then $p_i$ divides both $a-1$
and $a+1$ which is impossible (since $p_i>2$). Similarly, either
$m'\leq1$ or $m''\leq1$; otherwise, $2^{\min\{m',m''\}}$ divides $a-1$
and $a+1$ which is impossible.

Since $\gcd(a+1,n)\gcd(a-1,n)=jn$, we have that
$$2^{m'+m''}p_1^{s_1+t_1}\cdots p_k^{s_k+t_k}=j2^mp_1^{r_1}\cdots p_k^{r_k}.$$
Now, since for all $i$ either $t_i=0$ or $s_i=0$, we necessarily have
$s_i+t_i=r_i$. Dividing out, we get $2^{m'+m''}=j2^m$. We also know
that $m'\leq m$ and $m''\leq m$ and since either $m'\leq 1$ or
$m''\leq 1$, we have $m\leq m'+m''\leq m+1$. Dividing out we have
$j=2^{m'+m''-m}$ and so $1\leq j\leq 2$ as required.

Now, according to Theorem \ref{unorbits}, $N_a$ is the number of divisors of $\frac{\gcd(a-1,n)\gcd(a+1,n)}{n}$, and thus $N_a$ is the number of divisors of $j$ computed above. Therefore, $N_a=1$ if $j=1$ and $N_a=2$ if $j=2$. In either case, $N_a = |\zdiv{a+1}/\langle a-1 \rangle| \le 2$ proving Claim (2).

Finally, according to the proof of (1), we have $j=2$ if and only if
$m'=1$ and $m''=m$ or $m'=m$ and $m''=1$. In the first case, we have
that $2^m$ divides $a-1$, and in the second case we have that $2^m$
divides $a+1$. Therefore, $a \equiv \pm 1 \pmod{2^m}$.
\end{proof}

Theorem \ref{thm:equivinvols}, Theorem \ref{thm:squareroots1} and Remark \ref{rmk:computesqr} combined together allow us to compute the number of distinct equivalence classes of involutions in $\Aut_2(D_n)$ for fixed $n$ if we have the prime factorization of $n$. In fact, we get closed formula for the number of equivalence classes elements in $\Aut_2(D_n)$ as follows.

\begin{cor}
\label{cor:equivclassformula}
Suppose that $n\geq 1$ and $n=2^mp_1^{r_1}\cdots p_k^{r_k}$ where the $p_i$ are distinct odd primes. Then, the number of equivalence classes, $C_n$, of $\Aut_2(D_n)$ is given by
$$C_n=\begin{cases}
2^k & \mbox{ if $m=0$}\\
2^{k+1} & \mbox{ if $m=1$}\\
2^{k+2} & \mbox{ if $m=2$}\\
2^{k+3}-2^{k+1} & \mbox{ if $m\geq 3$}.
\end{cases}$$
\end{cor}

\begin{proof}
According to Theorem \ref{thm:equivinvols}, we know $N_a$ for every $a\in\mathcal{R}_n^2$. When $m=0$, $n$ is odd and we know that $N_a=1$ for all $a$. Thus, there is one equivalence class for every element of $\mathcal{R}_n^2$ proving the first formula. If $m=1,2$, then every $a\in\mathcal{R}_n^2$ is odd and so satisfies $a\equiv \pm 1\pmod{m}$. Thus for every element $a\in\mathcal{R}_n^2$, we get two equivalence classes for a total of $2\cdot|\mathcal{R}_n^2|$ equivalence classes. Finally, when $m\geq 3$, using the construction of the elements of $\mathcal{R}_n^2$ found in the proof in the appendix, we see that exactly half of these elements satisfy the condition $a\equiv\pm 1\pmod{m}$. Therefore, for half the elements in $\mathcal{R}_n^2$ we have $N_a=1$ and for the other half we have $N_a=2$ giving us a total of $2^{k+2}+(2^{k+2}-\frac{1}{2}\cdot2^{k+2})=2^{k+3}-2^{k+1}$ equivalence classes, proving the result.
\end{proof}

\begin{figure}
\centering
\includegraphics[width=5.5in]{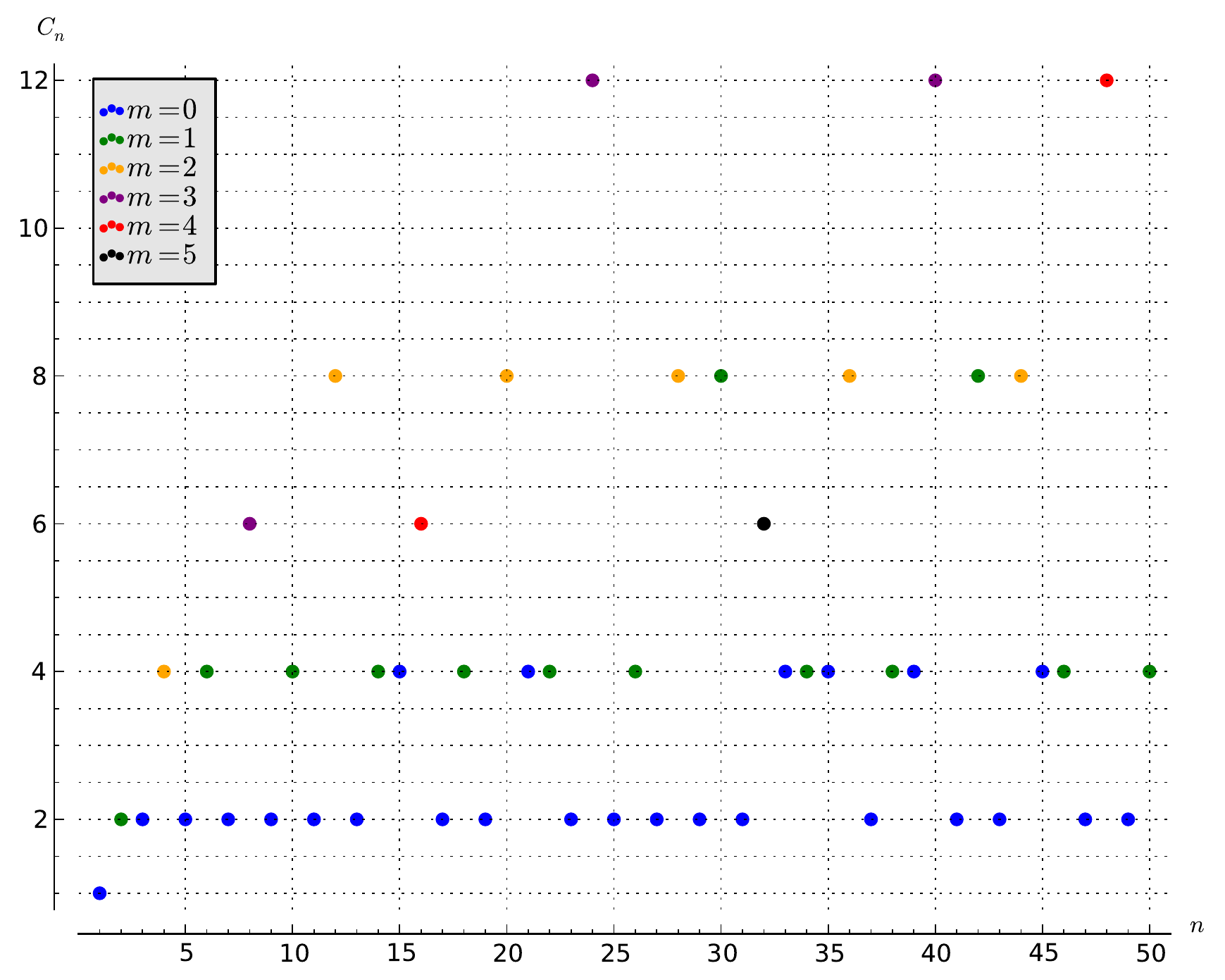}
\label{fig:classes-vs-n}
\caption{Number of equivalence class of involutions in $\Aut(D_n)$ vs. $n$ for $n \le 50$.}
\end{figure}

\begin{remark} According to the Online Encyclopedia of Integer Sequences, the sequence described in Corollary \ref{cor:equivclassformula} also counts the number (up to isomorphism) of groups of order $2n$ that have a subgroup isomorphic to $\mathbb Z_n$ \cite{oeisA147848}. To be more precise, we include a brief discussion of the natural bijection between these two objects. Suppose $\theta=ax+b$ is an involution. Then $w_a:\Z_2\to\Aut(\Z_n)=U_n$ given by $w_a(0)=1$ and $w_a(1)=-a$ is an injective homomorphism and induces a $\Z_2$-action on $\Z_n$. Next, we define $f:\Z_2\times\Z_2\to\Z_n$ by $f(c,d)=0$ if $c=0$ or $d=0$ and $f(1,1)=b$; it is easily shown that $f$ is a 2-cocycle if and only if $b\in\zdiv{a+1}$. Finally, it is well known (see \cite[Chapter 6.6]{weibel})) that 2-cocyles give rise to group extensions and two different 2-cocycles give isomorphic group extensions if they are cohomologous. A simple calculation shows that if $\theta=ax+b$ and $\theta'=ax+b'$, then the corresponding 2-cocycles are cohomologous if and only if $b-b'\in\langle a-1\rangle$, which matches the characterization of automorphisms being equivalent above.
\end{remark}

We now use our previous results from Section \ref{sec:results-aut} to
fully describe the sets $H_\theta$ and $Q_\theta$ when $\theta$ is an 
involution; moreover, in this situation, we are also interested in the set of twisted involutions
$$R=R_\theta=\{x\in G\mid \theta(x)=x^{-1}\}.$$

If $\theta=ax+b$ is an involution, a quick calculation gives 
\begin{equation}
\label{eqn:R}
R=\{r^k\mid k\in\zdiv{a+1}\}\cup \{r^ks\mid k(a-1)\equiv -b\pmod{n}\}.
\end{equation}

Following the results from Theorem \ref{thm:hqr} we describe $Q$ and $R\setminus Q$. 
\begin{cor}
\label{cor:Q-is-ZDiv}
If $\theta=ax+b$ is an involution, then $Q$ is a subgroup of $\langle
r \rangle$ and under the natural isomorphism $\psi: \langle r \rangle
\buildrel\cong\over\longrightarrow \Z_n$ we have 
\[ \psi(Q) = \begin{cases}
\langle a-1 \rangle, & \text{if} \quad b \in \langle a-1 \rangle \\
\zdiv{a+1}, & \text{otherwise.}
\end{cases} \]
Furthermore, 
$R \setminus Q = \{r^k\mid k\in\zdiv{a+1}\setminus\langle a-1\rangle\}\cup\{r^ks \mid k(a-1)\equiv -b\pmod{n}\}$
\end{cor}

\begin{proof}
  We proceed using Theorem \ref{thm:hqr} and
  applying Theorem \ref{thm:equivinvols}. Indeed, by Theorem \ref{thm:hqr} $\psi(Q)$ consists of the
  union of the two cosets $\langle a-1 \rangle$ and $-b + \langle a-1
  \rangle$. These cosets are
  the same if $-b \in \langle a-1 \rangle$ (which means $b\in\langle a-1\rangle$) and then $\psi(Q) = \langle a - 1 \rangle \le \zdiv{a+1}$, or they are
  distinct and by Theorem \ref{thm:equivinvols} part (1) there are
  only 2 cosets total so 
  $\langle a - 1 \rangle \cup (-b+\langle a - 1 \rangle) =
  \zdiv{a+1}$. In either case, $Q$ is a subgroup of $\langle r
  \rangle$ and $\psi(Q)$ is as described in the statement of the
  corollary. Finally, the result about $R \setminus Q$ is immediate
  from equation \eqref{eqn:R}.
\end{proof}

\begin{cor}
If $b\not\in\langle a-1\rangle$ then $R=Q$.
\end{cor}

\begin{remark}
We note that due to Theorem \ref{thm:equivinvols}, $\psi(Q)$ given in Corollary \ref{cor:Q-is-ZDiv} is almost always $\zdiv{a+1}$ for involutions. The only instance where $\psi(Q)\ne \zdiv{a+1}$ occurs when $|\zdiv{a+1}/\langle a-1 \rangle|=2$ and $b\in\langle a-1\rangle$.  We revisit $\theta_1=19x+18$ from Example \ref{autexamples}, which is an involution, noting that this is in fact a case where $b\in\langle a-1\rangle$.  Recalling that $Q_1=\{1,r^{18}\}$, we see that $\psi (Q_1)$ is indeed isomorphic to $\langle a-1\rangle=\langle 18\rangle < \mathbb Z_{36}$.  We also have that $R=R_1=\{1,r,r^3,r^5,\dots,r^{17},r^{18},r^{19},r^{21},\dots, r^{35}\}$, and thus $R_1\setminus Q_1=\{r,r^3,r^5,\dots,r^{35}\}$.
\end{remark}

Corollary \ref{cor:Q-is-ZDiv} is interesting because in the setting of
algebraic groups, the symmetric space $Q$ is almost never a subgroup; whereas here $Q$ is always a subgroup.

In the setting of algebraic groups, it is known that involutions 
$\theta_1 \sim \theta_2$ if and only if $H_{\theta_1}\cong
H_{\theta_2}$ (see \cite{Helm-Wang93}). 
Now we show that this result does not hold for finite groups.

\begin{cor}
  There exists $n$ and involutions $\theta_1,\theta_2\in\Aut(D_n)$
  such that $H_{\theta_1} = H_{\theta_2}$ but
  $\theta_1 \not\sim \theta_2$.
\end{cor}

\begin{proof}
  Let $n=8$ and $\theta_1=7x$ and $\theta_2=3x$. According to
  Proposition \ref{prop:equivalence}, $\theta_1\not\sim\theta_2$ because
  $7 \ne 3$ (in $U_8$). However, according to Theorem \ref{thm:hqr},
  $H_{\theta_1}=\{1, r^4, s, r^4s\}$ and $H_{\theta_2}=\{1, r^4, s,
  r^4s\}$.
\end{proof}

\section{The Infinite Dihedral Group $D_\infty$}

So far we have discussed the finite dihedral groups $D_n$. However, it turns out that there are similar 
results for the infinite dihedral group
\[ D_\infty=\langle r,s\mid s^2=1,\ rs=sr^{-1}\rangle. \] In this
case, the automorphisms are the affine linear transformations of $\Z$,
so are of the form $ax+b$ where $b\in\Z$ and $a\in\{\pm 1\}$. Then,
the congruences given in equation \eqref{congruences} become equations over the integers. In particular, it easy to show that
the only automorphisms of finite order, besides the identity, are the
involutions and they have the form $-x+b$ where $b\in\Z$.

\begin{prop}
\label{prop:D-infnity-invol}
The following hold:
\begin{enumerate}
\item If $\theta \in \Aut(D_\infty)$ has finite order, 
then $\theta = -x + b$ for some integer $b$.
\item $-x + b \sim -x + d$ if and only if $b \equiv d \pmod{2}$.
\end{enumerate}
\end{prop}

\begin{proof}
Part (1) follows from the equations \eqref{congruences} given in the proof of Proposition
\ref{prop:orderk}. For part (2), observe that Proposition
\ref{prop:equivalence} holds for $D_\infty$ and says that 
$-x + b \sim -x + d$ if and only if either $b - d \in 2\Z$ or $-b-d
\in 2\Z$, equivalently: $b \equiv d \pmod{2}$. 
\end{proof}

We see that in $D_\infty$ the situation is simple: the only automorphisms
of finite order are involutions and there are only two distinct
equivalence classes of involutions represented by $\chi_0 = -x = \autconj(s)$
(inner) and $\chi_1 = -x+1$ (outer). Note that $\chi_1$ represents the class of the diagram automorphism discussed in Example \ref{coxeter-example}.

The fixed group and symmetric space of an involution $\theta = \chi_i$ is
similarly easy to compute.
\begin{equation}
\label{eq:hqr-infty}
\begin{array}{lll}
H_{\chi_0} = \{1, s\},& Q_{\chi_0} = \langle r^2 \rangle,& R_{\chi_0} =
\langle r^2 \rangle \cup \{ s \} \\
H_{\chi_1} = \{ 1 \}, & Q_{\chi_1} = \langle r \rangle,& R_{\chi_1} =
\langle r \rangle.
\end{array}
\end{equation}

We note that the descriptions in equation \eqref{eq:hqr-infty} show that Corollary \ref{cor:Q-is-ZDiv} holds in the infinite case as well. However, in this case, the description is simpler because the description of $Q$ depends only on whether $b$ is even or odd. We also note that despite the fact that the two cases are different when viewing $Q$ as a subgroup of $D_\infty$, we always have $Q\cong \Z$.

\appendix
\section{Proof of Theorem \ref{thm:squareroots1}}
\label{sec:appendix}

\begin{proof}
Our goal is to determine the size of $|\mathcal{R}_n^2|$ for all $n$. First, suppose that $n=p^r$ where $p$ is an odd prime and $r\geq
  1$. Then, suppose $a\in\Z_{p^r}$ and that $a^2-1\equiv 0\pmod{p^r}$. Then
  we have that $(a+1)(a-1)\equiv 0$ (mod $p^r$) and so
  $(a+1)(a-1)=kp^r$ for some $k\in\Z$. Since $p$ is an odd prime, it
  is clear that both $a+1$ and $a-1$ cannot divide $p^d$ for any
  $d\geq 1$ and so we have that $a+1=p^r$ or $a-1=p^r$. This implies
  that $a=1$ or $a=n-1$, thus $|\mathcal{R}_{p^r}^2|=2$.  Next, it is
  clear that $|\mathcal{R}_2^2|=1$ and $|\mathcal{R}_4^2|=2$. Suppose that $n=2^m$ with $m\geq 3$. Next, suppose that
  $a\in\Z_n\setminus\{1,n-1\}$ with $a^2\equiv 1$ (mod $n$). Since $n$
  is even, we must have that $a$ is odd, so assume that $a=2k+1$ for
  some $0\leq k\leq 2^{m-1}$. Then $1\equiv a^2\equiv (2k+1)^2\equiv
  4k^2+4k+1$ (mod $n$) so that $4k(k+1)=l\cdot 2^m$ for some $l\geq
  1$. In particular, since $m\geq 3$, we have $k(k+1)=l\cdot 2^{m-2}$
  for some $m$. Either $k$ or $k+1$ is odd and thus cannot divide
  $2^{m-2}$.
\begin{enumerate}
\item[] Case 1: $k$ is even. So $k=h\cdot 2^{m-2}$ for some $1\leq
  h\leq 2$ (since $k\leq 2^{m-1})$. Then, if $h=1$, $a=2^{m-1}+1$. If
  $h=2$, then $a\equiv 1$ (mod $n$).
 
\item[] Case 2: $k$ is odd. So $k+1$ is even and $k+1=h\cdot 2^{m-2}$
  with $1\leq h\leq 2$. Then, if $h=1$, $a=2^{m-1}-1$. If $h=2$, then
  $a=2^m-1\equiv n-1\equiv -1$ (mod $n$).
\end{enumerate}

So, there are four possibilities for $a$:
$1,-1,2^{m-1}+1,2^{m-1}-1$. It is easy to check that these are all in
fact square roots of 1, showing that $|\mathcal{R}_{2^m}^2|=4$ if $m\geq 3$.

Finally, the result holds due to the fact that the Chinese Remainder
Theorem guarantees that if $n=2^mp_1^{r_1}\cdots p_k^{r_k}$ then
$$\Z_n\cong \Z_{2^m}\times\Z_{p_1^{r_1}}\times\cdots\times \Z_{p_k^{r_k}}.$$
Then, it is clear that any square root of $(1,...,1)$ (the identity
from the right hand side) must be of the form $(a_0,a_1,...,a_k)$
where $a_0\in\mathcal{R}_{2^m}^2$ and $a_i\in\mathcal{R}_{p_i^{r_i}}^n$ for $i\in\{1,...,k\}$. We have counted the possibilities for $a_i$ above and so we can multiply these
together to get all the possible choices for $a\in\Z_n$ with
$a^2\equiv 1$ (mod $n$). In particular, we get $|\mathcal{R}_n^2|=2^k$ if $m\in\{0,1\}$, $|\mathcal{R}_n^2|=2\cdot2^k$ if $m=2$ and $|\mathcal{R}_n^2|=2^2\cdot2^k$ if $m\geq 3$ as required.
\end{proof}

%%%%%%%%%%%%%%
%% References
%%%%%%%%%%%%%%

\bibliographystyle{amsplain}
\bibliography{dihedral_symmetric_spaces}

\end{document}